\documentclass{conm-p-l}
\usepackage{amssymb, euscript}
\newtheorem{theorem}{Theorem}[section]
\newtheorem{lemma}[theorem]{Lemma}
\newtheorem{corollary}[theorem]{Corollary}
\theoremstyle{definition}
\newtheorem{definition}[theorem]{Definition}

\theoremstyle{remark}

\numberwithin{equation}{section}

\begin{document}

\title[A Discrete Helgason-Fourier transform on
symmetric spaces ]{A Discrete Helgason-Fourier transform for
Sobolev and Besov  functions on noncompact symmetric spaces}

\author{Isaac Pesenson}
\address{Department of Mathematics, Temple University,
Philadelphia, PA 19122} \email{pesenson@temple.edu}
\date{January 1, 1994 and, in revised form, June 22, 1994.}

\begin{abstract}

Let $f$ be a Paley-Wiener function in the space  $L_{2}(X)$, where
$X$ is a symmetric space of noncompact type. It is shown that by
using the  values of $f$ on a sufficiently dense and separated set
of points of $X$ one can give an exact formula for the
Helgason-Fourier transform of $f$. In order to find a discrete
approximation to the Helgason-Fourier transform of a function from
a Besov space on $X$  we develop an approximation theory by
Paley-Wiener functions in $L_{2}(X)$.

\end{abstract}
\maketitle

\section{Introduction and main results}
\noindent
Let $f$ be a smooth function in the space  $L_{2}(X, dx)$, where
$X$ is a symmetric space of noncompact type and $dx$ is an
invariant measure. The notation $\widehat{f}$ will be used for the
Helgason-Fourier transform of $f$. The Helgason-Fourier transform
$\widehat{f}$ can be treated as a function on
$\mathbb{R}^{n}\times \mathcal{B}$ where $\mathcal{B}$ is a
certain compact homogeneous manifold and $n$ is the rank of $X$.
Moreover, $\widehat{f}$ belongs to the space
$$
L_{2}(\mathcal{R},d\mu)\equiv L_{2}\left(\mathbb{R}^{n}\times
\mathcal{B};|c(\lambda)|^{-2}d\lambda db \right),
$$
where $c(\lambda)$ is the Harish-Chandra's function, $d\lambda$ is
the Euclidean measure and $db$ is the normalized invariant measure
on $\mathcal{B}$, $\mathcal{R}=\mathbb{R}^{n}\times \mathcal{B},
d\mu=|c(\lambda)|^{-2}d\lambda db$. The notation
$\Pi_{\omega}\subset \mathbb{R}^{n}\times \mathcal{B}$ will be
used for the set of all points $(\lambda, b)\in
\mathbb{R}^{n}\times \mathcal{B}, \lambda\in \mathbb{R}^{n}, b\in
\mathcal{B},$ for which
$\sqrt{\left<\lambda,\lambda\right>}<\omega,$ where $\left<\cdot,
\cdot\right>$ is the Killing form.

The  Paley-Wiener space $PW_{\omega}(X), \omega>0,$ is defined as
the set of all functions in  $L_{2}(X)$ whose Helgason-Fourier
transform has support in $\Pi_{\omega}$ and belongs to the space
$$
\Lambda_{\omega}=L_{2}\left(\Pi_{\omega};|c(\lambda)|^{-2}d\lambda
db\right).
$$

It  is  shown that if  $f\in PW_{\omega}(X)$ is known only on a
sufficiently dense and separated set of points of $X$ then there
exists an exact formula for reconstruction the Helgason-Fourier
transform  $\widehat{f}$. In order to extend this result and to
find a discrete  approximation to the Helgason-Fourier transform
of a function from a Besov space on $X$ we develop an
approximation theory by Paley-Wiener functions in $L_{2}(X)$.

In the Section 2 we list some basic facts about harmonic Analysis
on symmetric spaces of noncompact type (see \cite{H1}-\cite{H3}).
One of the main results of the Section 3 is
 Corollary 3.1  which  says that for a fixed $\omega>0$ and a
sufficiently dense and separated  set of points
$Z_{\omega}=\{x_{j}\}, x_{j}\in X,$ there exists a set of
functions $\{\Theta_{x_{j}}\}, \Theta_{x_{j}}\in PW_{\omega}(X),$
such that for any $f\in PW_{\omega}(X)$ the following exact
formula holds
\begin{equation}
\widehat{f}= \sum_{x_{j}\in
Z_{\omega}}f(x_{j})\widehat{\Theta_{x_{j}}}.\label{DFT}
\end{equation}
This formula implies a quadrature rule which gives that for any
compact measurable set $U\subset X$

\begin{equation}
\int_{U}fdx= \sum_{x_{j}\in Z_{\omega}}f(x_{j})w_{j},\label{QR}
\end{equation}
for all $f\in PW_{\omega}(X)$. Here the  weights $w_{j}$ are given
by the formulas
$$
w_{j}=\int_{U}\Theta_{x_{j}}dx.
$$
In order to extend these results to non-Paley-Wiener functions we
consider  the following scheme. For a function $f\in L_{2}(X)$ we
consider its orthogonal projection on the space $PW_{\omega}(X)$
which is the function
 \begin{equation}
f_{\omega}=\mathcal{H}^{-1}\left(\chi_{\omega}\widehat{f}\right),
\end{equation}
where $\chi_{\omega}$ is the characteristic function of the set
$\Pi_{\omega}$ and $\mathcal{H}^{-1}$ is the inverse
Helgason-Fourier transform. It is clear that for a general
function $f\in L_{2}(X)$ the sum
$$
\widehat{f}_{\omega}=\sum_{x_{j}\in
Z_{\omega}}f_{\omega}(x_{j})\widehat{\Theta_{x_{j}}}
$$
gives just an approximation to the Helgason-Fourier $\widehat{f}$
of $f$  and a natural problem is to measure a degree of such
approximation.  If $Z_{\omega}=\{x_{j}\}, x_{j}\in X,$ is a set of
points for which the formula (\ref{DFT}) holds then by using the
Plancherel Theorem and the formula $\widehat{f_{\omega}}=
\chi_{\omega}\widehat{f}$ we obtain the following inequality
\begin{equation}
\Phi(f, Z_{\omega})\equiv\left\|\widehat{f}-\sum_{x_{j}\in
Z_{\omega}}f_{\omega}(x_{j})\widehat{\Theta_{x_{j}}}\right\|_{L_{2}(\mathcal{R},d\mu)}\leq
$$
$$\|\widehat{f}-\chi_{\omega}\widehat{f}\|_{L_{2}(\mathcal{R},d\mu)}+
\left\|\widehat{f_{\omega}}-\sum_{x_{j}\in
Z_{\omega}}f_{\omega}(x_{j})\widehat{\Theta_{x_{j}}}\right\|_{L_{2}(\mathcal{R},d\mu)}=
\|f-f_{\omega}\|_{L_{2}(X)}.\label{Error}
\end{equation}
This inequality shows that  the error of approximation of
$\widehat{f}$ for a general function $f\in L_{2}(X)$ by a sum
$\sum_{x_{j}\in
Z_{\omega}}f_{\omega}(x_{j})\widehat{\Theta_{x_{j}}}$ is
controlled by the best approximation $\mathcal{E}(f,\omega)$ of
$f\in L_{2}(X)$ by Paley-Wiener functions
\begin{equation}
\mathcal{E}(f,\omega)=\inf_{g\in
PW_{\omega}(X)}\|f-g\|_{L_{2}(X)}=\|f-f_{\omega}\|_{L_{2}(X)},
f\in L_{2}(X).\label{BA}
\end{equation}
The corresponding  approximation theory is developed in Sections 4
and 5. The main result of the Section 5 is Theorem 5.1 which
describes a rate of approximation of $f\in L_{2}(X)$ by
Paley-Wiener functions in terms of Besov spaces
$\mathbf{B}_{2,q}^{\alpha}(X), 1\leq q\leq \infty, \alpha>0 $. The
Besov spaces are described in terms of the one-parameter group
generated by the positive square root $\sqrt{-\Delta}$, where
$\Delta$ is the Laplace- Beltrami operator of an invariant metric
on $X$. We formulate here two particular cases of our main Theorem
5.2.

\begin{theorem} There exists a constant $C_{0}(X)$ and for every $\omega>0$
there exist a separated set of points $Z_{\omega}=\{x_{j}\}$ and a
set of functions $\left\{\Theta_{x_{j}}\right\}, \Theta_{x_{j}}\in
PW_{\omega}(X)$, as in (\ref{DFT}), such that
 for any $f$ in the Sobolev space
 $H^{\alpha}(X), \alpha> 0,$ the following  holds
\begin{equation}
\left(\int_{0}^{\infty}\left(\omega^{\alpha}\Phi(f,Z_{\omega})\right)^{2}
\frac{d\omega}{\omega}\right)^{1/2}\leq
C_{0}(X)\|f\|_{H^{\alpha}(X)},\label{Sob}
\end{equation}
where $\Phi(f,Z_{\omega})$ is defined in (\ref{Error}). Moreover,
if the following relation holds for an $0\leq \alpha\leq r, r\in
\mathbb{N},$
\begin{equation}
\left\|\left(I-e^{i
s\sqrt{-\Delta}}\right)^{r}f\right\|_{L_{2}(X)}= O(s^{\alpha}), s
\rightarrow 0,
\end{equation}
where $e^{i s\sqrt{-\Delta}}$ is the group generated by a positive
square root from the  operator $-\Delta$, then
\begin{equation}
\Phi(f; Z_{\omega}) =O(\omega^{-\alpha}), \omega\rightarrow
\infty.
\end{equation}
\end{theorem}

The results of the Section 5 are obtained as consequences of an
abstract  Direct Approximation Theorem 4.4 which is proved in the
Section 4. The Theorem 4.4 is an extension of the classical
results by Peetre and Sparr \cite{PS} about interpolation and
approximation spaces in abelian quasi-normed groups. The reason we
use the language of quasi-normed linear spaces is not because we
want to achieve a bigger generality but because this langauge
allows to treat simultaneously interpolation and approximation
spaces \cite{BL}, \cite{PS}. To be more specific: the two main
Theorems of this theory one of which gives a connection between
interpolation and approximation spaces and another one which is
known as the Power Theorem can be formulated only on the language
of quasi-normed linear spaces and not on the langauge of normed
linear spaces.

\section{Harmonic Analysis on symmetric spaces}
\noindent
A Riemannian symmetric space of the noncompact type is a
Riemannian manifold $X$ of the form $X=G/K$ where $G$ is a
connected semisimple Lie group with finite center and $K$ is a
maximal compact subgroup of $G$. The Lie algebras of the groups
$G$ and $K$ will be denoted respectively as $\textbf{g}$ and
$\textbf{k}$. The group $G$ acts on $X$ by left translations. If
$e$ is the identity in $G$  then the base point  $eK$ is denoted
by $0$. Every such $G$ admits Iwasawa decomposition $G=NAK$, where
the nilpotent Lie group $N$ and the abelian group $A$ have Lie
algebras $\textbf{n}$ and $\textbf{a}$ respectively. The dimension
of $\textbf{a}$ is known as the rank of $X$.
 The letter $M$ is usually used to denote the centralizer of $A$ in
$K$ and the letter $\mathcal{B}$ is commonly used for the
homogeneous space $K/M$.

Let $\textbf{a}^{*}$ be the real dual of $\textbf{a}$ and $W$ be
the Weyl's group. We denote by  $\Sigma$ will be the set of
restricted roots, and $\Sigma^{+}$ will be the set of all positive
roots. The notation $\textbf{a}^{+}$ has the following meaning
$$
\textbf{a}^{+}=\{h\in \textbf{a}|\alpha(h)>0, \alpha\in
\Sigma^{+}\}
$$
 and is known as positive Weyl's chamber. Let $\rho\in
 \textbf{a}^{*}$ is defined in a way that $2\rho$ is the sum of
 all  positive restricted roots. The Killing form $<,>$ on $\textbf{a}$
 defines a
 metric on $\textbf{a}$. By duality it defines a scalar product on
 $\textbf{a}^{*}$. We denote by  $\textbf{a}^{*}_{+}$ the set of
 $\lambda\in \textbf{a}^{*}$, whose dual belongs to
 $\textbf{a}^{+}$.
According to Iwasawa decomposition for every $g\in G$ there exists
a unique $A(g)\in \textbf{a}$ such that
$$g=n \exp A(g) k, k\in K, n\in N,
$$
 where $\exp :\textbf{a}\rightarrow A$ is the exponential map of
 the
 Lie algebra $\textbf{a}$ to Lie group $A$. On the direct product
 $X\times \mathcal{B}$ we introduce function with values in $\textbf{a}$
 using the formula
 \begin{equation}
 A(x,b)=A(u^{-1}g)
 \end{equation}
 where $x=gK, g\in G, b=uM, u\in K$.

For every $f\in C_{0}^{\infty}(X)$ the Helgason-Fourier transform
is defined by the formula
$$
\hat{f}(\lambda,b)=\int_{X}f(x)e^{(-i\lambda+\rho)(A(x,b))}dx,
$$
where $ \lambda\in \textbf{a}^{*}, b\in \mathcal{B}=K/M, $ and
$dx$ is a $G$-invariant measure on $X$. This integral can also be
expressed as an integral over the group $G$. Namely, if $b=uM,u\in
K$, then
\begin{equation}
\hat{f}(\lambda,b)=\int_{G}f(gK)e^{(-i\lambda+\rho)(A(u^{-1}g))}dg.
\end{equation}
The invariant measure on $X$ can be normalized so that the
following inversion formula holds for $f\in C_{0}^{\infty}(X)$
$$
f(x)=w^{-1}\int_{\textbf{a}^{*}\times
\mathcal{B}}\hat{f}(\lambda,b)e^{(i\lambda+\rho)(A(x,b))}|c(\lambda)|^{-2}d\lambda
db,
$$
where $w$ is the order of the Weyl's group and $c(\lambda)$ is the
Harish-Chandra's function, $d\lambda$ is the Euclidean measure on
$\textbf{a}^{*}$ and $db$ is the normalized $K$-invariant measure
on $\mathcal{B}$. This transform can be extended to an isomorphism
between the spaces $L_{2}(X,dx)$ and
$L_{2}(\textbf{a}^{*}_{+}\times \mathcal{B},
|c(\lambda)|^{-2}d\lambda db)$ and the Plancherel formula holds
true
$$
\|f\|=\left( \int_{\textbf{a}^{*}_{+}\times \mathcal{B}}|\hat{f}
(\lambda,b)|^{2}|c(\lambda)|^{-2}d\lambda db\right)^{1/2}.
$$

An analog of the Paley-Wiener Theorem is known which says in
particular that a Helgason-Fourier transform of a compactly
supported distribution is a function which is analytic in
$\lambda$.

Denote by $T_{x}(X)$ the tangent space of $X$ at a point $x\in X$
and let $ exp_ {x} $ :  $T_{x}(X)\rightarrow X$ be the exponential
geodesic map i.  e. $exp_{x}(u)=\gamma (1), u\in T_{x}(X)$ where
$\gamma (t)$ is the geodesic starting at $x$ with the initial
vector $u$ :  $\gamma (0)=x , \frac{d\gamma (0)}{dt}=u.$ In what
follows we assume that local coordinates  are defined by $exp$.

By using a uniformly bounded partition of unity
$\{\varphi_{\nu}\}$ subordinate to a cover of  $X$ of finite
multiplicity
$$
X=\bigcup_{\nu} \mathbf{B}(x_{\nu}, r),
$$
where $\mathbf{B}(x_{\nu}, r)$ is a metric ball at  $x_{\nu}\in X$
of radius $r$ we introduce Sobolev space $H^{\sigma}(X),
\sigma>0,$ as the completion of $C_{0}^{\infty}(X)$ with respect
to the norm
\begin{equation}
\|f\|_{H^{\sigma}(X)}=\left(\sum_{\nu}\|\varphi_{\nu}f\|^{2}
_{H^{\sigma}(\mathbf{B}(y_{\nu}, r))}\right) ^{1/2}.\label{SN}
\end{equation}

 The usual embedding Theorems for the spaces $H^{\sigma}(X)$
hold true.

The Killing form on $G$ induces an inner product on tangent spaces
of $X$. Using this inner product it is possible to construct
$G$-invariant Riemannian structure on $X$. The Laplace-Beltrami
operator of this Riemannian structure is denoted as $\Delta$.

It is known that the following formula holds
\begin{equation}
\widehat{\Delta
f}(\lambda,b)=-\left(\|\lambda\|^{2}+\|\rho\|^{2}\right)\hat{f}(\lambda,b),
f\in C_{0}^{\infty}(X),
\end{equation}
where
$\|\lambda\|^{2}=\left<\lambda,\lambda\right>,\|\rho\|^{2}=\left<\rho,\rho\right>,
\left<\cdot,\cdot\right>$ is the Killing form.

 It is also known that the operator $(-\Delta)$ is a self-adjoint positive
definite operator in the corresponding space $L_{2}(X,dx),$ where
$dx$ is the $G$-invariant measure. The regularity Theorem for the
Laplace-Beltrami operator $\Delta$ states that domains of the
powers
 $(-\Delta)^{\sigma/2}$ coincide with the Sobolev spaces
$H^{\sigma}(X)$ and  the norm (\ref{SN}) is equivalent to the
graph norm $\|f\|+\|(-\Delta)^{\sigma/2}f\|$ (see \cite{T3}, Sec.
7.4.5.) Moreover, since the operator $\Delta$ is invertible in
$L_{2}(X)$ the Sobolev norm is also equivalent to the norm
$\|(-\Delta)^{\sigma/2}f\|.$

\bigskip

\section{Paley-Wiener functions and their Discrete Helgason-Fourier transform}
\noindent
\begin{definition}
We will say that $f\in L_{2}(X,dx)$ belongs to the class
$PW_{\omega}(X)$ if its Helgason-Fourier transform has compact
support in the sense that $\hat{f}(\lambda,b)=0$ a. e. for
$\|\lambda\|>\omega$. Such functions will  be also called
$\omega$-band limited.
\end{definition}
Using the spectral resolution of identity $P_{\lambda}$ we define
the unitary group of operators by the formula
$$
e^{it\Delta}f=\int_{0}^{\infty}e^{it\tau}dP_{\tau}f, f\in
L_{2}(X), t\in \mathbb{R}.
$$

Let us introduce the operator
\begin{equation}
\textbf{R}_{\Delta}^{\sigma}f=\frac{\sigma}{\pi^{2}}\sum_{k\in\mathbb{Z}}\frac{(-1)^{k-1}}{(k-1/2)^{2}}
e^{i\left(\frac{\pi}{\sigma}(k-1/2)\right)\Delta}f, f\in L_{2}(X),
\sigma>0.\label{Riesz1}
\end{equation}
 Since  $\left\|e^{it\Delta}f\right\|=\|f\| $ and
\begin{equation}
\frac{\sigma}{\pi^{2}}\sum_{k\in\mathbb{Z}}\frac{1}{(k-1/2)^{2}}=\sigma,\label{id}
 \end{equation}
 the series in (\ref{Riesz1}) is convergent and it shows that
 $\textbf{R}_{\Delta}^{\sigma}$ is a bounded operator in $L_{2}(X)$
 with the norm $\sigma$:
\begin{equation}
 \|\textbf{R}_{\Delta}^{\sigma}f\|\leq \sigma\|f\|, f\in
L_{2}(X).\label{Riesznorm}
 \end{equation}

 The next theorem contains generalizations of several results
 from the classical harmonic analysis (in particular  the Paley-Wiener theorem)
  and it  follows essentially
from our more general results in \cite{P1}, \cite{P2},
\cite{P3}(see also \cite{A}, \cite{Pa}).
\begin{theorem} Let $f\in L_2(X)$. Then the following
statements are equivalent:
\begin{enumerate}
\item $f\in PW_{\omega}(X)$;
 \item $f\in
H^{\infty}(X)=\bigcap_{k=1}^{\infty}H^{k}(X), $ and for all $s\in
\mathbb{R}_{+}$ the following Bernstein inequality holds:
\begin{equation}
\|\Delta^{s}f\|\leq( \omega^{2}+\|\rho\|^{2})^{s}\|f\|;\label{B}
\end{equation}

\item $f\in H^{\infty}(X)$ and the
following Riesz interpolation formula holds
\begin{equation}
\Delta^{n}f=\left(\textbf{R}_{\Delta}^{
\omega^{2}+\|\rho\|^{2}}\right)^{n}f, n\in \mathbb{N};
\label{Rieszn}
\end{equation}

\item  For every $g\in
L_{2}(X)$ the function $t\mapsto
\left<e^{it\Delta}f,g\right>, t\in \mathbb{R}^{1}$,
 is bounded on the real line and has an extension to the complex
plane as an entire function of the exponential type
$\omega^{2}+\|\rho\|^{2}$;

\item The abstract-valued
function $t\mapsto e^{it\Delta}f$  is bounded on the real line and has an
extension to the complex plane as an entire function of the
exponential type $\omega^{2}+\|\rho\|^{2}$;

\item The solution $u(t),
t\in \mathbb{R}^{1}$, of the Cauchy problem
$$
i\frac{\partial u(t)}{\partial t}=\Delta u(t), u(0)=f,
i=\sqrt{-1},
$$
has a holomorphic extension $u(z)$ to the
complex plane $\mathbb{C}$ satisfying
$$
\|u(z)\|_{L_{2}(X)}\leq e^{(\omega^{2}+\|\rho\|^{2})|\Im
z|}\|f\|_{L_{2}(X)}.
$$
\end{enumerate}
\label{PW}

\end{theorem}

Now we give new characterizations of the space $PW_{\omega}(X)$.
We will need the following Lemma which was proved in \cite{P1}.
\begin{lemma}If for some $f\in H^{\infty}(X)$ and a certain $\sigma>0$ the upper bound
\begin{equation}
\sup _{k\in N
}(\sigma^{-k}\|\Delta^{k}f\|)=C(f,\sigma)<\infty,\label{251}
\end{equation}
is finite,  then $C(f,\sigma)\leq \|f\|$ and the following
inequality holds
$$
\|\Delta^{k}f\|\leq \sigma^{k}\|f\|, k\in \mathbb{N}.
$$
 \label{Lem}
\end{lemma}For a vector $f\in PW_{\omega }(X)$ the notation
$\omega_{f}$ will be used for a smallest positive number such that
$\Pi_{\omega_{f}}$ contains the support of the Helgason-Fourier
transform $\widehat{f}$. The following Theorem gives a new
characterization of the Paley-Wiener spaces.
\begin{theorem}
A vector $f\in L_{2}(X)$ belongs to the space $PW_{\omega_{f}}(X),
0<\omega_{f}<\infty,$ if and only if $f$ belongs to the set
$H^{\infty}(X)$, the limit
$$
 \lim_{k\rightarrow \infty}\|\Delta^{k}f\|^{1/k}
$$
exists and
\begin{equation}
\lim_{k\rightarrow
\infty}\|\Delta^{k}f\|^{1/k}=\omega_{f}^{2}+\|\rho\|^{2}.\label{limitcond}
\end{equation}

\end{theorem}
\begin{proof}

If $f\in PW_{\omega_{f}}(X)$ then $f$ is obviously in
$H^{\infty}(X)$ and for $\widehat{f}$ we have
  $$
\|\Delta^{k}f\|^{1/k}=\left(
\int_{\|\lambda\|<\omega_{f}}\int_{\mathcal{B}}(\|\lambda\|^{2}+\|\rho\|^{2})
^{k} |\widehat{f}(\lambda,b)|^{2}|c(\lambda)|^{-2}d\lambda
db\right)^{1/2k}\leq $$ $$(\omega_{f}^{2}+\|\rho\|^{2})
\|f\|^{1/k},
$$
which shows that
\begin{equation}
\overline{\lim}_{k\rightarrow\infty}\|\Delta^{k}f\|^{1/k}\leq
\omega_{f}^{2}+\|\rho\|^{2}.\label{23}
\end{equation}

Now, assume that
\begin{equation}
\underline{\lim}_{k\rightarrow\infty}\|\Delta^{k}f\|^{1/k}=\sigma^{2}+\|\rho\|^{2}<
\omega_{f}^{2}+\|\rho\|^{2},
\end{equation}
which means that there exists a sequence $k_{j}\rightarrow\infty$
for which
\begin{equation}
\underline{\lim}_{k_{j}\rightarrow\infty}\|\Delta^{k_{j}}f\|^{1/k_{j}}=
\sigma^{2}+\|\rho\|^{2}< \omega_{f}^{2}+\|\rho\|^{2}.\label{ASS}
\end{equation}
Note that the following inequality holds
\begin{equation}
\|\Delta^{m}f\|\leq \pi\|\Delta^{k}f\|^{m/k}\|f\|^{1-m/k}, 0\leq
m\leq k.\label{KS}
\end{equation}
Indeed, for any $h\in L_{2}(X)$ the Kolmogorov-Stein inequality
gives
$$
\left\|\left(\frac{d}{dt}\right)^{m}\left<e^{t\Delta}f,h\right>\right\|_{
C(R^{1})}\leq
\pi\left\|\left(\frac{d}{dt}\right)^{k}\left<e^{t\Delta}f,h\right>\right\|_{
C(R^{1})}^{m/k}\left\|\left<e^{t\Delta}f,h\right>\right\|_{C(R^{1})}^{1-m/k},
$$
or
$$
\left\|\left<e^{t\Delta}\Delta^{m}f,h\right>\right\|_{
C(R^{1})}\leq
\pi\left\|\left<e^{t\Delta}\Delta^{k}f,h\right>\right\|_{
C(R^{1})}^{m/k}\left\|\left<e^{t\Delta}f,h\right>\right\|_{C(R^{1})}^{1-m/k}.
$$
Applying the Schwartz inequality we obtain
$$
\left\|\left<e^{it\Delta}\Delta^{m}f,h\right>\right\|_{
C(R^{1})}\leq\pi\|h\|^{m/k}\|\Delta^{k}f\|^{m/k}\|h\|^{1-m/k}\|f\|^{1-m/k}
=
$$
$$
\pi\|\Delta^{k}f\|^{m/k}\|f\|^{1-m/k}\|h\|.
$$
When $t=0$ it gives
$$
\left|\left<\Delta^{m}f,h\right>\right|\leq\pi\|\Delta^{k}f\|^{m/k}\|f\|^{1-m/k}\|h\|.
$$
By choosing $h$ such that
$\left|\left<\Delta^{m}f,h\right>\right|=\|\Delta^{m}f\|$ and
$\|h\|=1$ we obtain (\ref{KS}). The assumption (\ref{ASS}) and the
inequality (\ref{KS}) imply that the quantity
\begin{equation}
\sup _{k\in N }(\sigma
^{2}+\|\rho\|^{2})^{-k}\|\Delta^{k}f\|=C(f,\sigma),\label{252}
\end{equation}
is finite and the previous Lemma \ref{Lem} gives the inequality
$$
\|\Delta^{k}f\|\leq (\sigma ^{2}+\|\rho\|^{2})^{k}\|f\|, k\in
\mathbb{N}.
$$
According to the Theorem \ref{PW} the last  inequality shows that
$f\in PW_{\sigma}(X)$. Since $ \sigma<\omega_{f}$  this
contradicts to the definition of $\omega_{f}$.

Conversely, assume that the following  holds
\begin{equation}
\lim_{k\rightarrow
\infty}\|\Delta^{k}f\|^{1/k}=\sigma^{2}+\|\rho\|^{2}\label{assump100}
\end{equation}
for a certain $\sigma>0$. It would imply
\begin{equation}
\sup_{k\in
\mathbb{N}}(\sigma^{2}+\|\rho\|^{2})^{-k}\|\Delta^{k}f\|<C(f,
\sigma)
\end{equation}
for some $C(f,\sigma)>0$ and by Lemma \ref{Lem} one would have
$$
\|\Delta^{k}f\|\leq (\sigma ^{2}+\|\rho\|^{2})^{k}\|f\|, k\in
\mathbb{N}.
$$
It shows that $f\in PW_{\sigma}(X)$ and there exists an
$\omega_{f}\leq \sigma$. But as it was just shown, this fact
implies (\ref{limitcond}) which together with (\ref{assump100})
gives $\omega_{f}=\sigma$.
 The Theorem is proved.
\end{proof}
 The above Lemma and the proof of the Theorem imply two other
 characterizations of the Paley-Wiener spaces.
 \begin{corollary} The following holds true:
 \begin{enumerate}
\item a function$f\in L_{2}(X)$ belongs to $PW_{\omega}(X)$ if and
only if $f\in H^{\infty}(X)$ and  the upper bound
\begin{equation}
\sup _{k\in \mathbb{N} }\left((\omega
^{2}+\|\rho\|^{2})^{-k}\|\Delta^{k}f\|\right)<\infty
\end{equation}
is finite,

\item a function $f\in L_{2}(X)$ belongs to $PW_{\omega}(X)$ if
and only if $f\in H^{\infty}(X)$ and
\begin{equation}
\underline{\lim}_{k\rightarrow\infty}\|\Delta^{k}f\|^{1/k}=\omega^{2}+\|\rho\|^{2}<\infty.
\end{equation}
In this case $\omega=\omega_{f}$.
 \end{enumerate}

 \end{corollary}

In \cite{P1} the following Lemma was proved.
\begin{lemma}
There exists a natural number $N=N(X)\in \mathbb{N}$ such that for
any sufficiently small $r>0$ there
 exists a set of points $\{x_{j}\}$ from $X$ with the following
 properties:

\begin{enumerate}
\item the balls $B(x_{j}, r/4)$ are disjoint, \item  the balls
$B(x_{j}, r/2)$ form a cover of $X$,
\item  the multiplicity of the
cover by balls $B(x_{j}, r)$ is not greater $N.$

\end{enumerate}

\end{lemma}
We will use notation $Z=Z( r, N)$ for any set of points
$\{x_{j}\}\in X$ which satisfies the properties (1)- (3) from the
last Lemma and we will call such set a metric $(r,N)$-lattice of
$X$.

If $\delta_{x_{j}}$ is a Dirac distribution at a point $x_{j}\in
X$ then according to the inversion formula for the
Helgason-Fourier transform we have
$$
\left<\delta_{x_{j}},f\right>= w^{-1}\int_{\textbf{a}^{*}\times
\mathcal{B}}\hat{f}(\lambda,b)e^{(i\lambda+\rho)(A(x_{j},b))}|c(\lambda)|^{-2}d\lambda
db.
$$

It implies that if $f\in L_{2}(X)$ then the action on
$\hat{f}(\lambda, b)$ of the Helgason-Fourier transform $
\widehat{\delta_{x_{j}}}$ of $ \delta_{x_{j}}$ is given by the
formula
\begin{equation}
\hat{f}(\lambda, b) \rightarrow \left<\widehat{\delta_{x_{j}}},
\hat{f}\right> =w^{-1}\int_{\textbf{a}^{*}\times
\mathcal{B}}e^{(i\lambda+\rho)(A(x_{j},b))}\hat{f}(\lambda,
b)|c(\lambda)|^{-2}d\lambda db.
\end{equation}

We introduce the notation $k^{\omega}_{x_{j}}$ for a function
which is a restriction of the smooth function
$\widehat{\delta_{x_{j}}}$ to the set $\Pi_{\omega}$:
$$
k^{\omega}_{x_{j}}=\widehat{\delta_{x_{j}}}|_{\Pi_{\omega}},
$$
and
$$
\left<k^{\omega}_{x_{j}},\widehat{f}\right>=w^{-1}\int_{\textbf{a}^{*}\times
\mathcal{B}}\chi_{\omega}(\lambda)e^{(i\lambda+\rho)(A(x_{j},b))}\hat{f}(\lambda,
b)|c(\lambda)|^{-2}d\lambda db, f\in L_{2}(X),
$$
where $\chi_{\omega}$ is the characteristic function of the set
$\Pi_{\omega}$.
\begin{theorem}
There exists a constant $ c(X)$ such that for any given
$\omega>0$, for every $(r,N)$-lattice  $Z_{\omega}=Z(r, N)$ with
$$
r=c(X)(\omega^{2}+\|\rho\|^{2})^{-1/2},
$$
the set of functions
 $\{k^{\omega}_{x_{j}}\}, x_{j}\in Z_{\omega},$ is a frame in the space
\begin{equation}
\Lambda_{\omega}=L_{2}\left(\Pi_{\omega};
|c(\lambda)|^{-2}d\lambda db\right)
\end{equation}
 and there exists a  frame
 $\{\Theta_{x_{j}}\}$ in the space $PW_{\omega}(X)$
such that every $\omega$-band limited function $f\in
PW_{\omega}(X)$ can be reconstructed from a set of samples
$f(x_{j})=\left<\delta_{x_{j}},f\right>$ by using the formula
\begin{equation}
f=\sum_{x_{j}\in Z_{\omega}}f(x_{j})\Theta_{x_{j}}\label{decomp}.
\end{equation}

\end{theorem}

\begin{proof}

It was shown  in \cite{P1} that for any $k>d/2$  there exist
constants $C_{1}(X)>0, C_{2}(X)>0,  r_{0}(X)>0,$
 such that for any  for any $k>d/2$, any $0<r<r_{0}(X)$ and any $(r,N)$-lattice
 $Z=Z(r,N)$ the following inequality holds true

\begin{equation}
\|f\|\leq C_{1}(X)\left\{r^{d/2}\left(\sum_{x_{j}\in Z}
|f(x_{j})|^{2}\right)^{1/2}+r^{k}\|\Delta^{k/2}f\|\right\},
\end{equation}
where $f\in H^{k}(X), k>d/2,$ and
$$\left(\sum _{x_{j}\in Z}|f(x_{j})|^{2}\right)^{1/2}\leq
 C_{2}(X)\|f\|_{H^{k}(X)}, f\in H^{k}(X), k>d/2.
 $$

Along with the Bernstein inequality (\ref{B}) it implies that
there exist positive constants $c(X), c_{1}(X), c_{2}(X),$ such
that for every $\omega>0$, every lattice $Z_{\omega}=Z(r , N)$
with $r= c(X)(\omega^{2}+\|\rho\|^{2})^{-1/2}$ and every $f\in
PW_{\omega}(X)$ the following inequalities hold true

\begin{equation}
c_{1}(X)\left(\sum_{x_{j}\in
Z_{\omega}}\left|f(x_{j})\right|^{2}\right)^{1/2}\leq
r^{-d/2}\|f\|_{2} \leq c_{2}(X)\left(\sum_{x_{j}\in Z_{\omega}}
|f(x_{j})|^{2}\right)^{1/2}.
\end{equation}

An application of the Plancherel formula gives that there are
constants $A_{1}(X)>0,A_{2}(X)>0,$ such that for any $f\in
PW_{\omega}(X)$

\begin{equation}
A_{1}(X)\|\widehat{f}\|_{\Lambda_{\omega}}\leq
\left(\sum_{x_{j}\in
Z_{\omega}}\left|\left<\widehat{\delta_{x_{j}}},\widehat{f}\right>_{\Lambda_{\omega}}\right|^{2}\right)^{1/2}\leq
A_{2}(X)\|\widehat{f}\|_{\Lambda_{\omega}},\label{FI}
\end{equation}
where $\left<\cdot,\cdot\right>$ is the scalar product in the
space $\Lambda_{\omega}= L_{2}\left(\Pi_{\omega};
|c(\lambda)|^{-2}d\lambda db\right)$.

Since
$k^{\omega}_{x_{j}}=\widehat{\delta_{x_{j}}}|_{\Pi_{\omega}}\in
\Lambda_{\omega}$ and since for $f\in PW_{\omega}(X)$
$$
\left<\widehat{\delta_{x_{j}}},\widehat{f}\right>_{\Lambda_{\omega}}=
\left<k^{\omega}_{x_{j}},\widehat{f}\right>_{\Lambda_{\omega}},
$$
the inequalities (\ref{FI}) show that the set of functions
$\{k^{\omega}_{x_{j}}\}, x_{j}\in Z_{\omega},$ is a frame in the
space $\Lambda_{\omega}$.

It is known \cite{DS} that to construct a dual frame one has to
consider the so called frame operator
$$
F(\widehat{f})=\sum_{x_{j}\in
Z_{\omega}}\left<k^{\omega}_{x_{j}},\widehat{f}\right>_{\Lambda_{\omega}}
k^{\omega}_{x_{j}}.
$$
One can show that the frame operator $F$ is invertible and the
formula
\begin{equation}
\widehat{\Theta_{x_{j}}}=F^{-1}k^{\omega}_{x_{j}}
\end{equation}
gives a dual frame $\Lambda_{\omega}$. A reconstruction formula of
a function $f$ can be written in terms of the dual frame as
\begin{equation}
\widehat{f}=\sum_{x_{j}\in
Z_{\omega}}\left<\widehat{\Theta_{x_{j}}},\widehat{f}\right>_{\Lambda_{\omega}}
k^{\omega}_{x_{j}}= \sum_{x_{j}\in
Z_{\omega}}\left<k^{\omega}_{x_{j}},\widehat{f}\right>_{\Lambda_{\omega}}\widehat{\Theta_{x_{j}}},\label{Frecon}
\end{equation}
where inner product is taken in the space $\Lambda_{\omega}$.

Functions $\widehat{\Theta_{x_{j}}}$ belong to the space
$\Lambda_{\omega}= L_{2}\left(\Pi_{\omega};
|c(\lambda)|^{-2}d\lambda db\right)$. By extending them by zero
outside of the set $\Pi_{\omega}$ we can treat them as functions
in
$L_{2}=L_{2}(\textbf{a}^{*}_{+}\times\mathcal{B},|c(\lambda)|^{-2}d\lambda
db )$ and then by taking the inverse Helgason-Fourier transform
$\mathcal{H}^{-1}$ of both sides of the  formula (\ref{Frecon}) we
obtain the formula (\ref{decomp}) of our Theorem because for $f\in
PW_{\omega}(X)$ the Plancherel Theorem gives
 $$
\left<k^{\omega}_{x_{j}},\widehat{f}\right>_{\Lambda_{\omega}}=
\left<\widehat{\delta_{x_{j}}},
\widehat{f}\right>_{L_{2}(\mathcal{R},d\mu)}=
\left<\delta_{x_{j}},f\right>=f(x_{j}).
$$

\end{proof}

In the classical case when $X=\mathbb{R}$ is the one-dimensional
Euclidean space we have
\begin{equation}
\widehat{\delta_{x_{j}}}(\lambda)= e^{i x_{j}\lambda}.
\end{equation}
In this situation  the Theorem  means that the complex
exponentials $e^{i x_{j}\lambda}, x_{j}\in Z(r, N )$ form a frame
in the space $L_{2}([-\omega, \omega])$. Note that in the case of
a uniform point-wise sampling in the space $L_{2}(\mathbb{R})$
this result gives the classical sampling formula
$$
f(t)= \sum f(\gamma n\Omega )\frac{\sin(\omega (t-\gamma n\Omega
))}{\omega (t-\gamma n\Omega )}, \Omega =\pi /\omega, \gamma<1,
$$
with a certain oversampling.

The formula (\ref{Frecon}) can be treated as a Discrete
Helgason-Fourier transform in the following sense.

\begin{corollary}
There exists a constant $ c(X)$ such that for any given
$\omega>0$, for every $(r,N)$-lattice  $Z_{\omega}=Z(r, N)$ with
$$
r=c(X)(\omega^{2}+\|\rho\|^{2})^{-1/2},
$$
there exist functions $ \left \{\Theta_{x_{j}}\right\}, x_{j}\in
Z_{\omega},$ in the space $PW_{\omega}(X)$ such that

\begin{equation}
\widehat{f}= \sum_{x_{j}\in
Z_{\omega}}f(x_{j})\widehat{\Theta_{x_{j}}},
\end{equation}
for all $f\in PW_{\omega}(X)$.
\end{corollary}

The formula (\ref{decomp}) can be used to obtain the following
quadrature rule for functions from $PW_{\omega}(X), \omega>0.$

\begin{corollary}
There exists a constant $ c(X)$ such that for any given
$\omega>0$, for every $(r,N)$-lattice  $Z_{\omega}=Z(r, N)$ with
$$
r=c(X)(\omega^{2}+\|\rho\|^{2})^{-1/2},
$$
and for every compact $U\subset X$ there exist a set of numbers
$\{w_{j}\}$  such that

\begin{equation}
\int_{U}fdx= \sum_{x_{j}\in Z_{\omega}}f(x_{j})w_{x_{j}},
\end{equation}
for all $f\in PW_{\omega}(X)$. The weights $w_{x_{j}}$ are given
by the formulas
$$
w_{x_{j}}=\int_{U}\Theta_{x_{j}}dx,
$$
where $\Theta_{x_{j}}$ are the functions from (\ref{decomp}).
\end{corollary}

As another  consequence we obtain  the following quadrature
formula on the Fourier transform-side.

 \begin{corollary}
  For any measurable set
$V\subset \Pi_{\omega}$ there exist weights
$$
\upsilon_{x_{j}}=\int_{V}\widehat{\Theta_{x_{j}}}(\lambda,b)|c(\lambda)|^{-2}d\lambda
db
$$
such that
$$
\int_{V}\hat{f}(\lambda,b)|c(\lambda)|^{-2}d\lambda
db=\sum_{x_{j}\in Z_{\omega}}f(x_{j})\upsilon_{x_{j}}
$$
for all $f\in PW_{\omega}(X)$.
\end{corollary}

\section{A Direct  Approximation Theorem for
 quasi-normed linear spaces}
\noindent
The goal of the section is to establish certain  connections
between interpolation spaces and approximation spaces which will
be used later to develop an approximation theory on a symmetric
space $X$. The general theory of interpolation spaces can be found
in \cite{BL}, \cite{PS}, \cite{KPS}, \cite{T3}. The notion of
approximation spaces and their relations to interpolations spaces
can be found in \cite{BL}, Ch. 3 and 7, and in \cite{PS}. As it
was explained in the Introduction we use the language of
quasi-normed linear spaces  not because we want to achieve a
bigger generality but because this langauge allows to treat
simultaneously interpolation and approximation spaces.

Let $E$ be a linear space. A quasi-norm $\|\cdot\|_{E}$ on $E$ is
a real-valued function on $E$ such that for any $f,f_{1}, f_{2}\in
E$ the following holds true

\begin{enumerate}
\item $\|f\|_{E}\geq 0;$

\item $\|f\|_{E}=0  \Longleftrightarrow   f=0;$

\item $\|-f\|_{E}=\|f\|_{E};$

\item $\|f_{1}+f_{2}\|_{E}\leq C_{E}(\|f_{1}\|_{E}+\|f_{2}\|_{E}),
C_{E}>1.$

\end{enumerate}

We say that two quasi-normed linear spaces $E$ and $F$ form a
pair, if they are linear subspaces of a linear space $\mathcal{A}$
and the conditions $\|f_{k}-g\|_{E}\rightarrow 0,$ and
$\|f_{k}-h\|_{F}\rightarrow 0, f_{k}, g, h \in \mathcal{A}, $
imply equality $g=h$. For a such pair $E,F$ one can construct a
new quasi-normed linear space $E\bigcap F$ with quasi-norm
$$
\|f\|_{E\bigcap F}=\max\left(\|f\|_{E},\|f\|_{F}\right)
$$
and another one $E+F$ with the quasi-norm
$$
\|f\|_{E+ F}=\inf_{f=f_{0}+f_{1},f_{0}\in E, f_{1}\in
F}\left(\|f_{0}\|_{E}+\|f_{1}\|_{F}\right).
$$

All quasi-normed spaces $H$ for which $E\bigcap F\subset H \subset
E+F$ are called intermediate between $E$ and $F$. A group
homomorphism $T: E\rightarrow F$ is called bounded if
$$
\|T\|=\sup_{f\in E,f\neq 0}\|Tf\|_{F}/\|f\|_{E}<\infty.
$$
One says that an intermediate quasi-normed linear space $H$
interpolates between $E$ and $F$ if every bounded homomorphism $T:
E+F\rightarrow E+F$ which is a bounded homomorphism of $E$ into
$E$ and a bounded homomorphism of $F$ into $F$ is also a bounded
homomorphism of $H$ into $H$.

On $E+F$ one considers the so-called Peetere's $K$-functional
\begin{equation}
K(f, t)=K(f, t,E, F)=\inf_{f=f_{0}+f_{1},f_{0}\in E, f_{1}\in
F}\left(\|f_{0}\|_{E}+t\|f_{1}\|_{F}\right).\label{K}
\end{equation}
The quasi-normed linear space $(E,F)^{K}_{\theta,q}, 0<\theta<1,
0<q\leq \infty,$ or $0\leq\theta\leq 1,  q= \infty,$ is introduced
as a set of elements $f$ in $E+F$ for which
\begin{equation}
\|f\|_{\theta,q}=\left(\int_{0}^{\infty}
\left(t^{-\theta}K(f,t)\right)^{q}\frac{dt}{t}\right)^{1/q}.\label{Knorm}
\end{equation}

It turns out that $(E,F)^{K}_{\theta,q}, 0<\theta<1, 0\leq q\leq
\infty,$ or $0\leq\theta\leq 1,  q= \infty,$ with the quasi-norm
(\ref{Knorm})  interpolates between $E$ and $F$. The following
Reiteration Theorem is one of the main results of the theory (see
\cite{BL}, \cite{PS}, \cite{KPS}, \cite{T3}).
\begin{theorem}
Suppose that $E_{0}, E_{1}$ are complete intermediate quasi-normed
linear spaces for the pair $E,F$.  If $E_{i}\in
\mathcal{K}(\theta_{i}, E, F)$ which means
$$
K(f,t,E,F)\leq Ct^{\theta_{i}}\|f\|_{E_{i}}, i=0,1,
$$
where $ 0\leq \theta_{i}\leq 1, \theta_{0}\neq\theta_{1},$ then
$$
(E_{0},E_{1})^{K}_{\eta,q}\subset (E,F)^{K}_{\theta,q},
$$
where $0<q<\infty, 0<\eta<1,
\theta=(1-\eta)\theta_{0}+\eta\theta_{1}$.

 If for the same pair
$E,F$ and the same $E_{0}, E_{1}$  one has $E_{i}\in
\mathcal{J}(\theta_{i}, E, F)$  that means
$$
\|f\|_{E_{i}}\leq C\|f\|_{E}^{1-\theta_{i}}\|f\|_{F}^{\theta_{i}},
i=0,1,
$$
where $ 0\leq \theta_{i}\leq 1, \theta_{0}\neq\theta_{1},$ then
$$
(E,F)^{K}_{\theta,q},\subset (E_{0},E_{1})^{K}_{\eta,q},
$$
where $0<q<\infty, 0<\eta<1,
\theta=(1-\eta)\theta_{0}+\eta\theta_{1}$.

\end{theorem}

It is important to note that in all cases which will be considered
in the present article the space $F$ will be continuously embedded
as a subspace into $E$. In this case (\ref{K}) can be introduced
by the formula
$$
K(f, t)=\inf_{f_{1}\in
F}\left(\|f-f_{1}\|_{E}+t\|f_{1}\|_{F}\right),
$$
which implies the inequality
\begin{equation}
K(f, t)\leq \|f\|_{E}.
\end{equation}
This inequality can be used to show  that the norm (\ref{Knorm})
is equivalent to the norm
\begin{equation}
\|f\|_{\theta,q}=\|f\|_{E}+\left(\int_{0}^{\varepsilon}
\left(t^{-\theta}K(f, t)\right)^{q}\frac{dt}{t}\right)^{1/q},
\varepsilon>0,
\end{equation}
for any positive $\varepsilon$.

Let us introduce another functional on $E+F$, where $E$ and $F$
form a pair of quasi-normed linear spaces
$$
\mathcal{E}(f, t)=\mathcal{E}(f, t,E, F)=\inf_{g\in F,
\|g\|_{F}\leq t}\|f-g\|_{E}.
$$

\begin{definition}The approximation space $\mathcal{E}_{\alpha,q}(E, F),
0<\alpha<\infty, 0<q\leq \infty $ is a quasi-normed linear spaces
of all $f\in E+F$ with the following quasi-norm
\begin{equation}
\left(\int_{0}^{\infty}\left(t^{\alpha}\mathcal{E}(f,
t)\right)^{q}\frac{dt}{t}\right)^{1/q}.
\end{equation}
\end{definition}

For a general quasi-normed linear spaces  $E$ the notation
$(E)^{\rho}$ is used for a quasi-normed linear spaces whose
quasi-norm is $\|\cdot\|^{\rho}$.

The following Theorem describes relations between interpolation
and approximation spaces (see  \cite{BL}, Ch. 7).
\begin{theorem}
If $\theta=1/(\alpha+1)$ and $r=\theta q, $ then
$$
(\mathcal{E}_{\alpha,q}(E, F))^{\theta}=(E,F)^{K}_{\theta,q}.
$$
\end{theorem}
The following important result is known as the Power Theorem (see
\cite{BL}, Ch. 7).
\begin{theorem}
Suppose that the following relations satisfied:  $\nu=\eta
\rho_{1}/\rho,$  $\rho=(1-\eta)\rho_{0}+\eta \rho_{1}, $ and
$q=\rho r$ for $ \rho_{0}>0, \rho_{1}>0.$ Then, if $0<\eta< 1, 0<
r\leq \infty,$ the following equality holds true
$$
\left((E)^{\rho_{0}}, (F)^{\rho_{1}}\right)^{K}_{\eta,
r}=\left((E,F)^{K}_{\nu, q}\right)^{\rho}.
$$
\end{theorem}

The Theorem we prove next represents a very general version of a
Jackson-type Theorem (see \cite{KP}, \cite{P4}, \cite{P5}).  Such
type results are known as Direct Approximation Theorems.

\begin{theorem}
 Suppose that $\mathcal{P}\subset F\subset E$ are quasi-normed
linear spaces and $E$ and $F$ are complete. Suppose that there
exist $C>0$ and $\beta >0$ such that for any $f\in F$ the
following Jackson-type inequality verified
\begin{equation}
t^{\beta}\mathcal{E}(t,f,\mathcal{P},E)\leq C\|f\|_{F},\label{dir}
t>0,
\end{equation}
 then the following embedding holds true
\begin{equation}
(E,F)^{K}_{\theta,q}\subset \mathcal{E}_{\theta\beta,q}(E,
\mathcal{P}), 0<\theta<1, 1<q\leq \infty.
\end{equation}

\end{theorem}\label{intthm}

\begin{proof}

It is known (\cite{BL}, Ch.7)  that for any $s>0$, for
\begin{equation}
t=K_{\infty}(f,
s)=K_{\infty}(f,s,\mathcal{P},E)=\inf_{f=f_{1}+f_{2}, f_{1}\in
\mathcal{P}, f_{2}\in E}\max(\|f_{1}\|_{\mathcal{P}},
s\|f_{2}\|_{E})\label{t}
\end{equation}
the following inequality holds
\begin{equation}
s^{-1}K_{\infty}(f, s)\leq \lim_{\tau\rightarrow
t-0}\inf\mathcal{E}(f, \tau,E,\mathcal{P})\label{lim}.
\end{equation}
Since
\begin{equation}
K_{\infty}(f,s)\leq K(f, s)\leq 2K_{\infty}(f, s),\label{equiv}
\end{equation}
the Jackson-type inequality (\ref{dir}) and the inequality
 (\ref{lim}) imply
\begin{equation}
s^{-1}K(f,s,\mathcal{P},E)\leq Ct^{-\beta}\|f\|_{F}.
\end{equation}
The equality (\ref{t}), and inequality (\ref{equiv}) imply the
estimate
\begin{equation}
t^{-\beta}\leq 2^{\beta}
\left(K(f,s,\mathcal{P},E)\right)^{-\beta}
\end{equation}
which along with the previous inequality gives the estimate
$$
K^{1+\beta}(f,s,\mathcal{P},E)\leq C s \|f\|_{F}
$$
which in turn imply the inequality
\begin{equation}
K(f,s,\mathcal{P},E)\leq C s^{\frac{1}{1+\beta}}
\|f\|_{F}^{\frac{1}{1+\beta}}.\label{class1}
\end{equation}
At the same time one has
\begin{equation}
K(f, s,\mathcal{P},E)= \inf_{f=f_{0}+f_{1},f_{0}\in \mathcal{P},
f_{1}\in E}\left(\|f_{0}\|_{\mathcal{P}}+s\|f_{1}\|_{E}\right)\leq
s\|f\|_{E},\label{class2}
\end{equation}
for every $f$ in $E$. The inequality (\ref{class1}) means that the
quasi-normed linear space $(F)^{\frac{1}{1+\beta}}$ belongs to the
class $\mathcal{K}(\frac{1}{1+\beta}, \mathcal{P}, E)$ and
(\ref{class2}) means that the quasi-normed linear space $E$
belongs to the class $\mathcal{K}(1, \mathcal{P}, E)$. This fact
allows to use the Reiteration Theorem to obtain the embedding
\begin{equation}
\left((F)^{\frac{1}{1+\beta}},E\right)^{K}_{\frac{1-\theta}{1+\theta
\beta},q(1+\theta \beta)}\subset \left(\mathcal{P},
E\right)^{K}_{\frac{1}{1+\theta \beta},q(1+\theta \beta)}
\end{equation}
for every $0<\theta<1, 1<q<\infty$. But the space on the left is
the space
$$
\left(E,(F)^{\frac{1}{1+\beta}}\right)^{K}_{\frac{\theta(1+\beta)}{1+\theta
\beta},q(1+\theta \beta)},
$$
which is according to the Power Theorem is the space
$$
\left((E,F)^{K}_{\theta, q}\right)^{\frac{1}{1+\theta \beta}}.
$$
All these results along with the equivalence of  interpolation and
approximation spaces give  the embedding
$$
\left(E,F\right)^{K}_{\theta,q}\subset \left(\left(\mathcal{P},
E\right)^{K}_{\frac{1}{1+\theta \beta},q(1+\theta
\beta)}\right)^{1+\theta \beta}=\mathcal{E}_{\theta \beta,
q}(E,\mathcal{P}),
$$
which proves  the Theorem.

\end{proof}

\section{Approximation by Paley-Wiener functions on noncompact symmetric spaces}
\noindent
The Helgason-Fourier transform can be treated as a unitary
operator from the space $L_{2}(X)$ onto the space
$L_{2}\left(\textbf{a}_{+}^{*}, |c(\lambda)|^{-2}d\lambda \right)$
of abstract-valued functions
$$
\widehat{f}(\lambda,\cdot):\textbf{a}_{+}^{*}\rightarrow
L_{2}(\mathcal{B},db).
$$
Define the support of $\widehat{f}(\lambda,\cdot)$ as the
complement of the maximal open set $\mathcal{U}\subset
\textbf{a}^{*}$ such that $\widehat{f}(\lambda,\cdot)=0$ for
almost all $\lambda\in \mathcal{U}$.

 Consider the space
$$
PW(X)=\bigcup_{t>0}PW_{t}(X),
$$
which is  a quasi-normed linear space with respect to the
quasi-norm
\begin{equation}
\|f\|_{PW(X)}=\sup\left\{\|\lambda\|^{2}=\left<\lambda,\lambda\right>:\lambda\in
supp \hat{f}(\lambda,\cdot)\right\},
\end{equation}
where $\left<\cdot,\cdot\right>$ is the Killing form on
$\textbf{a}^{*}$.

The corresponding approximation functional takes the form
$$
\mathcal{E}(f, t)=\mathcal{E}(f,t,PW(X), L_{2}(X))=\inf_{g\in
PW_{t}(X)}\|f-g\|, f\in L_{2}(X).
$$

The next goal is to introduce  Besov spaces on $X$  in terms of
the group $e^{is\sqrt{-\Delta}}$ generated in $L_{2}(X)$ by the
operator $\sqrt{-\Delta}$. Consider a difference operator of order
$r\in \mathbb{N}$ as
$$
\left(I-e^{is\sqrt{-\Delta}}\right)^{r}f=
(-1)^{r+1}\sum^{r}_{k=0}(-1)^{k-1}C^{k}_{r}e^{iks\sqrt{-\Delta}}f,
f\in L_{2}(X).
$$
and the modulus of continuity which is defined as
$$
\Omega_{r}(f,s)=\sup_{\tau\leq s}\left\|\left(I-e^{i\tau
\sqrt{-\Delta}}\right)^{r}f\right\|.
$$

The Besov space $\mathbf{B}_{2,q}^{\alpha}(X), 1\leq q\leq \infty,
\alpha>0, $ as an  interpolation space between $L_{2}(X)$ and
Sobolev space $H^{r}(X)$,
\begin{equation}
\mathbf{B}_{2,q}^{\alpha}(X)=\left(L_{2}(X),
H^{r}(X)\right)_{\alpha/r,q}^{K},\label{interpol}
\end{equation}
where $r\in \mathbb{N}, 0<\alpha<r, 1\leq q<\infty,$ or $0\leq
\alpha\leq r, q=\infty.$

The fact that the Sobolev space $H^{r}(X)$ is the domain of a
self-adjoint operator $(-\Delta)^{r/2}$ implies (see \cite{BL},
\cite{BB}, \cite{KPS}, \cite{T3}) that this definition
(\ref{interpol}) is independent on $r$ and it is the reason that
$r$ does not appear on the left side of the last formula.
Furthermore the Besov norms can be given by the
following formulas (see \cite{BL}, \cite{BB},  \cite{KPS},
\cite{T3})
\begin{equation}
 \|f\|_{\mathbf{B}_{2,q}^{\alpha}(X)}=\|f\|+
 \left(\int_{0}^{\infty}\left(s^{-\alpha}\Omega_{r}(f,s)\right)^{q}
\frac{ds}{s} \right)^{1/q},\label{BN1}
\end{equation}
where $ 0<\alpha<r, 1\leq q<\infty$ or
\begin{equation}
\|f\|_{\mathbf{B}_{2,\infty}^{\alpha}(X)}=\|f\|+
\sup_{0<s<\infty}\left(s^{-\alpha}\Omega_{r}(f,s)\right),\label{BN2}
\end{equation}
where $0\leq \alpha\leq r, q= \infty$. Note that many other
descriptions of Besov spaces on complete manifolds and symmetric
spaces were given in \cite{S1}- \cite{S3} and \cite{T1}-
\cite{T3}.

Our goal is to prove the following Theorem which gives description
of Besov spaces in terms of the best approximations by
Paley-Wiener functions.
\begin{theorem}
The following embedding holds true
\begin{equation}
\mathbf{B}_{2,q}^{\alpha}(X)\subset \mathcal{E}_{\alpha,q}(PW(X),
L_{2}(X)),\label{T3.5}
\end{equation}
where $   0<\alpha<\infty,  1\leq q\leq\infty . $ In other words
there exists a constant $C(X)$ such that for all $f\in
\mathbf{B}_{2,q}^{\alpha}(X)$ the following inequality holds
\begin{equation}
\left(\int_{0}^{\infty}\left(s^{\alpha}\mathcal{E}(f,
s)\right)^{q}\frac{ds}{s}\right)^{1/q}\leq C(X)\left(\|f\|+
 \left(\int_{0}^{\infty}\left(s^{-\alpha}\Omega_{r}(f,s)\right)^{q}
\frac{ds}{s} \right)^{1/q}\right),
\end{equation}
for any $0<\alpha<r\in \mathbb{N}$ if $1\leq q<\infty$ and any
$0<\alpha\leq r\in \mathbb{N}$ if $1\leq q\leq \infty$.
\end{theorem}
\begin{proof}
We have to prove  the following embedding
\begin{equation}
(L_{2}(X),H^{r}(X))^{K}_{\alpha/r,q}\subset
\mathcal{E}_{\alpha,q}( PW(X), L_{2}(X)),
\end{equation}
where $  \alpha<r\in \mathbb{N}, 1<q\leq \infty. $
 In order to be able to apply our Theorem 4.4
  we are going to verify the Jackson-type inequality (\ref{dir}).

   Let $\chi_{t}$ be the
characteristic function of the set $\Pi_{t}$. According to the
Plancherel Theorem
\begin{equation}
\mathcal{E}(f, t,PW(X), L_{2}(X))= \inf_{g\in PW_{t}(X)}\|f-g\|_{
L_{2}(X)}=
$$
$$
\inf_{g\in
PW_{t}(X)}\|\hat{f}-\hat{g}\|_{L_{2}(\textbf{a}^{*}\times\mathcal{B},d\mu)},\label{inf}
\end{equation}
where $d\mu=|c(\lambda)|^{-2}d\lambda db$.  But it is obvious that
the $\inf$ in the last formula is achieved exactly when
$\hat{g}=\chi_{t} \hat{f}$. Since the Sobolev space $H^{r}(X)$ is
the domain of
 $(-\Delta)^{r/2}$ we obtain the following inequalities for every  $f\in H^{r}(X)$
\begin{equation}
\mathcal{E}(f, t,PW(X), L_{2}(X))=
$$
$$
\left(\int_{\Pi_{t}}
  (\|\lambda\|^{2}+\|\rho\|^{2})
^{-r}(\|\lambda\|^{2}+\|\rho\|^{2})
^{r}|\hat{f}(\lambda,b)|^{2}|c(\lambda)|^{-2}d\lambda
db\right)^{1/2}\leq
$$
$$
(t^{2}+\|\rho\|^{2}) ^{-r/2}\|f\|_{r}\leq t^{-r}\|f\|_{r}, r\in
\mathbb{N},\label{estofba}
\end{equation}
that shows that the Jackson-type inequality (\ref{dir}) is
satisfied and $\beta=r$. Thus, by Theorem 4.4 and (\ref{interpol}) we obtain
(\ref{T3.5}) and the claim follows.
\end{proof}

As a consequence of this result in the case $q=\infty$ we obtain
the following Corollary.
\begin{corollary}
If for a function $f\in L_{2}(X)$ the following relation hods
\begin{equation}
\Omega_{r}(f,s)= O(s^{\alpha}),  s \rightarrow 0.
\end{equation}
for an $0<\alpha\leq r\in \mathbb{N},$ then the following
relations holds
\begin{equation}
\mathcal{E}(f, s,PW(X), L_{2}(X))= O(s^{-\alpha}), s\rightarrow
\infty.
\end{equation}

\end{corollary}

Let $c(X)$ will be the same constant as in Theorem 3.4. We
introduce the functional
\begin{equation}
\Phi(f; Z_{\omega})=\left\|\widehat{f}-\sum_{x_{j}\in
Z_{\omega}}f_{\omega}(x_{j})\widehat{\Theta_{x_{j}}}\right\|_{L_{2}(\textbf{a}^{*}\times\mathcal{B},d\mu)},\omega>0,\label{func}
\end{equation}
where $Z_{\omega}=Z(r,N)$ is any lattice with
$r=c(X)(\omega^{2}+\|\rho\|^{2})^{-1/2}$ and
$\left\{\widehat{\Theta_{x_{j}}}\right\}$ is the corresponding
frame in the space $
\Lambda_{\omega}=L_{2}\left(\Pi_{\omega};|c(\lambda)|^{-2}d\lambda
db\right) $.

 The following Theorem is a consequence
of (\ref{Error}), (\ref{BA}), and  Theorem 5.1.

\begin{theorem}There exists a constant $C_{0}(X)$ and for any $\omega>0$
there exists a sequence of $(r,N)$-lattices $Z_{\omega}=Z(r,N)$
 with $r=c(X)(\omega^{2}+\|\rho\|^{2})^{-1/2}$ such that for any
$f\in \mathbf{B}_{2,q}^{\alpha}(X),   0<\alpha<\infty,  1\leq
q\leq\infty $ the following inequality holds
\begin{equation}
\left(\int_{0}^{\infty}\left(\omega^{\alpha}
    \Phi(f; Z_{\omega})  \right)^{q}\frac{d\omega}{\omega}\right)^{1/q}\leq
C_{0}(X)\|f\|_{\mathbf{B}_{2,q}^{\alpha}(X)},
\end{equation}
where functional $ \Phi(f; Z_{\omega})$ is defined in
(\ref{func}).
\end{theorem}

 Since the Sobolev space
$H^{s}(X)$ is the domain of the self-adjoint operator
$(-\Delta)^{s/2}$ in the Hilbert space $L_{2}(X)$
 the general theory of interpolation spaces \cite{KPS}, \cite{T3},
 implies the isomorphism
$$
\mathbf{B}_{2,2}^{\alpha}(X)=\left(L_{2}(X),H^{r}(X)\right)_{\alpha/r,
2}^{K}=H^{\alpha}(X).
 $$
  Using this fact
 we obtain the Theorem 1.1 as a
consequence of  Theorem 5.2 and Corollary 5.1.

\section{Acknowledgment}

I would like to thank the anonymous referee for constructive
suggestions.
\bibliographystyle{amsalpha}

\end{document}